\documentclass[12pt]{amsart}
\usepackage{amssymb,latexsym,eufrak,amsmath,amscd, graphicx}

\usepackage{amsfonts}
\usepackage{amscd}

\renewcommand{\phi}{\varphi}
\renewcommand{\epsilon}{\varepsilon}
\renewcommand{\theta}{\vartheta}

\def\AAA{{\mathbf A}}

\def\QQ{{\mathbf Q}}

\def\cH{\mathcal{H}}
\def\cI{\mathcal{I}}
\def\cJ{\mathcal{J}}

\def\cO{\mathcal{O}}
\def\cT{\mathcal{T}}
\def\cM{\mathcal{M}}

\def\fra{\mathfrak{a}}
\def\frb{\mathfrak{b}}
\def\frc{\mathfrak{c}}

\def\frm{\mathfrak{m}}

\def\frp{\mathfrak{p}}
\def\frq{\mathfrak{q}}

\def\Z{{\mathbf Z}}

\def\A{{\mathbf A}}

\def\Q{{\mathbf Q}}

\def\O{\mathcal{O}}

 \DeclareMathOperator{\Spec}{Spec}
 
 \DeclareMathOperator{\lct}{lct}
 
 \DeclareMathOperator{\ord}{ord}

\DeclareMathOperator{\mult} {mult}

\DeclareMathOperator{\Div} {div}

\def\.{\cdot}
\def\~{\widetilde}
\def\^{\widehat}

\def\o{\circ}

\def\rat{\dashrightarrow}

\newcommand{\llbracket}{[\negthinspace[}
\newcommand{\rrbracket}{]\negthinspace]}

\newtheorem{lemma}{Lemma}[section]
\newtheorem{theorem}[lemma]{Theorem}
\newtheorem{corollary}[lemma]{Corollary}
\newtheorem{proposition}[lemma]{Proposition}

\theoremstyle{definition}
\newtheorem{definition}[lemma]{Definition}
\newtheorem{remark}[lemma]{Remark}

\theoremstyle{remark}
\newtheorem*{remark*}{Remark}
\newtheorem*{note*}{Note}

\title{Shokurov's ACC Conjecture
for log canonical thresholds on smooth varieties}

\author[T.~de Fernex]{Tommaso de Fernex}
\address{Department of Mathematics, University of Utah, 155 South 1400 East,
Salt Lake City, UT 48112-0090, USA} \email{{\tt
defernex@math.utah.edu}}

\author[L. Ein]{Lawrence~Ein}
\address{Department of Mathematics, University of
Illinois at Chicago, 851 South Morgan Street (M/C 249),
Chicago, IL 60607-7045, USA}
\email{{\tt ein@math.uic.edu}}

\author[M. Musta\c{t}\u{a}]{Mircea~Musta\c{t}\u{a}}
\address{Department of Mathematics, University of Michigan,
Ann Arbor, MI 48109, USA}
\email{{\tt mmustata@umich.edu}}

\thanks{2000\,\emph{Mathematics Subject Classification}.
 Primary 14E15; Secondary 14B05, 14E30.
\newline
The first author was partially supported by NSF
CAREER grant DMS-0847059,
the second author
  was partially supported by NSF grant DMS-0700774,
  and the third author was partially supported by
  NSF grant DMS-0758454, and
  by a Packard Fellowship}
\keywords{Log canonical threshold, ascending chain condition,
inversion of adjunction, $\frm$-adic approximation, connectedness theorem}

\begin{document}

\begin{abstract}
Shokurov conjectured that the set of all log canonical thresholds
on varieties of bounded dimension satisfies the ascending chain condition.
In this paper we prove that the conjecture holds for log canonical thresholds
on smooth varieties and, more generally, on locally complete intersection varieties
and on varieties with quotient singularities.
\end{abstract}

\maketitle

\markboth{T. DE FERNEX, L.~EIN AND M.~MUSTA\c{T}\u{A}}
{LOG CANONICAL THRESHOLDS ON SMOOTH VARIETIES}

\section{Introduction}

Let $k$ be an algebraically closed field of characteristic zero.
Log canonical varieties
are varieties with mild singularities that provide the most general context
for the Minimal Model Program. More generally, one considers the log canonicity condition
on pairs $(X,\fra^t)$, where $\fra$ is a proper ideal sheaf on $X$
(most of the times, it is the ideal
of an effective Cartier divisor), and $t$ is a nonnegative real number.
Given a log canonical variety $X$ over $k$, and a proper nonzero ideal sheaf $\fra$ on $X$,
one defines the {\it log canonical threshold} $\lct(\fra)$ of the pair $(X,\fra)$.
This is the largest number $t$ such that the pair $(X,\fra^t)$ is log canonical.
One makes the convention
$\lct(0) = 0$ and $\lct(\O_X) = \infty$.
The log canonical threshold
is a fundamental invariant in birational geometry, see for example \cite{Kol2},
\cite{EM2}, or Chapter~9 in \cite{positivity}.

Shokurov's ACC Conjecture \cite{Sho} says that the set of all log canonical thresholds
on varieties of any fixed dimension
satisfies the ascending chain condition, that is, it contains no
infinite strictly increasing sequences. This conjecture attracted considerable interest
due to its implications to the Termination of Flips Conjecture
(see \cite{Birkar} for a result in this direction).
The first unconditional results on sequences of log canonical thresholds
on smooth varieties of arbitrary
dimension have been obtained in \cite{dFM}, and they were
subsequently reproved and strengthened in \cite{Kol1}.

The main goal of this paper is to prove Shokurov's ACC Conjecture for
log canonical thresholds on smooth varieties and, more generally,
on varieties that are locally complete intersection (l.c.i. for short).
Our first result deals with the smooth case.

\begin{theorem}
\label{thm:intro:T_n^sm}
For every $n$, the set
$$
\cT_n^{\rm sm} :=
\{\lct(\fra)\mid \text{$X$ is smooth, $\dim X = n$, $\fra\subsetneq\cO_X$} \}
$$
of log canonical thresholds on smooth varieties of dimension $n$
satisfies the ascending chain condition.
\end{theorem}

As we will see, every log canonical threshold on a
variety with quotient singularities can be written as a
log canonical threshold on a smooth variety
of the same dimension. Therefore for every $n$ the set
$$
\cT_n^{\rm quot} :=
\{\lct(\fra)\mid \text{$X$ has quotient singularities,
$\dim X = n$, $\fra\subsetneq\cO_X$} \}
$$
is equal to $\cT_n^{\rm sm}$, and thus
the ascending chain condition also holds for
log canonical thresholds on varieties with quotient singularities.

In order to extend the result to log canonical thresholds on
locally complete intersection varieties,
we consider a more general version of log canonical thresholds.
Given a variety $X$ and an ideal sheaf
$\frb$ on $X$ such that the pair $(X,\frb)$ is log canonical, for every nonzero ideal
sheaf $\fra\subsetneq\cO_X$ we define the \emph{mixed log canonical
threshold} $\lct_{(X,\frb)}(\fra)$ to be the largest number $c$ such that the pair
$(X,\frb\cdot\fra^c)$ is log canonical. Note that when $\frb=\cO_X$, this is nothing but $\lct(\fra)$.
Again, one sets $\lct_{(X,\frb)}(0) = 0$ and $\lct_{(X,\frb)}(\cO_X)=\infty$.
The following is our main result.

\begin{theorem}
\label{thm:intro:M_n^lci}
For every $n$, the set
$$
\cM_n^{\rm l.c.i.} :=
\{\lct_{(X,\frb)}(\fra)\mid \text{$X$ is l.c.i., $\dim X = n$,
$\fra,\frb\subseteq\cO_X$, $\fra \ne \cO_X$, $(X,\frb)$ log canonical\,} \}
$$
of mixed log canonical thresholds on l.c.i. varieties of dimension $n$
satisfies the ascending chain condition.
\end{theorem}

By restricting to the case $\frb = \O_X$, we obtain the following immediate corollary.

\begin{corollary}\label{cor:intro:T_n^lci}
For every $n$, the set
$$
\cT_n^{\rm l.c.i.} :=
\{\lct(\fra)\mid \text{$X$ is log canonical and l.c.i., $\dim X = n$, $\fra\subsetneq\cO_X$} \}
$$
of log canonical thresholds on log canonical l.c.i. varieties of dimension $n$
satisfies the ascending chain condition.
\end{corollary}

We will use Inversion of Adjunction (in the form treated in \cite{EM3}) to
reduce Theorem~\ref{thm:intro:M_n^lci} to the analogous statement
in which $X$ ranges over smooth varieties.
More precisely, we show that all sets
$$
\cM_n^{\rm sm} :=
\{\lct_{(X,\frb)}(\fra)\mid \text{$X$ is smooth, $\dim X = n$,
$\fra,\frb\subseteq\cO_X$, $\fra \ne \cO_X$, $(X,\frb)$ log canonical} \}
$$
satisfy the ascending chain condition.
It follows by Inversion of Adjunction that
every mixed log canonical threshold of the
form $\lct_{(X,\frb)}(\fra)$, with $\fra$ and $\frb$
ideal sheaves on an l.c.i. variety $X$, can be expressed as a mixed log canonical threshold
on a (typically higher dimensional) smooth variety.
This is the step that requires us to work with mixed log canonical thresholds.
The key observation that makes this approach work is that
if $X$ is an l.c.i. variety with log canonical singularities, then
$\dim_kT_xX\leq 2\dim X$ for every $x\in X$. This implies that the above reduction
to the smooth case keeps the dimension of the ambient variety bounded.

The proofs of the above results use a general method of associating to a sequence of ideals
of polynomials over a field $k$, an ideal of power series over a field extension of $k$.
The original construction considered in \cite{dFM} is a standard application
of nonstandard methods, and relies on the use
of ultrafilters. This construction
was subsequently replaced in \cite{Kol1} by a purely algebro-geometric
construction, that gives a \emph{generic limit} by using 
a sequence of $\frm$-adic approximations and field extensions.
The two constructions are similar in nature, and either construction
can be employed to prove the results of this paper.
We chose to present the proofs using the second construction,
which is geometrically more explicit.

A key ingredient is the following effective $\frm$-adic
semicontinuity property for log canonical thresholds
(that we will only use in the case when $X = \A^n$ and
$E$ lies over a point of $\A^n$).

\begin{theorem}\label{thm:intro:m-adic-semicont:ideals}
Let $X$ be a log canonical variety, and let $\fra \subsetneq \O_X$ be a proper ideal.
Suppose that $E$ is a prime divisor over $X$ computing $\lct(\fra)$, and consider
the ideal sheaf $\frq := \{ h \in \O_X \mid \ord_E(h) > \ord_E(\fra) \}$.
If $\frb\subseteq\cO_X$ is an ideal such that $\frb+\frq=\fra+\frq$, then 
after possibly restricting to an open neighborhood of the center of $E$, we have
$\lct(\frb)=\lct(\fra)$.
\end{theorem}

This result (for principal ideals) was first proven by Koll\'ar 
in \cite{Kol1} using deep results in the Minimal Model Program from \cite{BCHM}
and a theorem on Inversion of Adjunction from \cite{Kawakita}.
We give an elementary proof of 
Theorem~\ref{thm:intro:m-adic-semicont:ideals}
which only uses the Connectedness Theorem
of Shokurov and Koll\'ar (see Theorem~7.4 in \cite{Kol2}).
We note that in the case of a divisor $E$ with zero-dimensional center, Koll\'{a}r's proof
extends to cover also ideals in a power series ring, and this fact is important for
his approach.
In fact, as we will see, this version can be formally deduced from the statement of 
Theorem~\ref{thm:intro:m-adic-semicont:ideals}
(see Corollary~\ref{formal_case}).

It is interesting to observe how, in the end, all the results of
this paper only rely on basic facts in birational geometry, such
as Resolution of Singularities and the Connectedness Theorem
and, for the l.c.i. case, on Inversion of Adjunction.
We expect however that new ideas and more sophisticated techniques
will be necessary to tackle the ACC Conjecture
in its general formulation.

\smallskip

\noindent{\bf Acknowledgment}. We are grateful to Shihoko Ishii and Angelo Vistoli for
useful discussions and correspondence, and to J\'{a}nos Koll\'{a}r for his comments
and suggestions on previous versions of our work. Furthermore, as we have already mentioned,
two key ideas we use in this paper come from Koll\'{a}r's article \cite{Kol1}.

\section{Generalities on log canonical thresholds}

Let $k$ be a field of characteristic zero.
In what follows $X$ will be either a normal and $\QQ$-Gorenstein
variety over $k$, or $\Spec\left(k\llbracket x_1,\ldots,x_n\rrbracket\right)$.

We recall the definition of log canonical threshold in a slightly
more general version, and discuss some of the properties that will be needed later.
For the basic facts about log canonical pairs in the setting of algebraic varieties, see
\cite{Kol2} or Chapter~9 in \cite{positivity}, while for the case of the
spectrum of a formal power series ring we refer to \cite{dFM}. The key point is that by
\cite{Temkin}, log resolutions exist also in
the latter case, and therefore the usual theory
of log canonical pairs carries through.

Suppose that $X$ is as above. Let $\fra$ and $\frb$ be nonzero coherent sheaves
of ideals on $X$ with $\fra \ne \cO_X$,
and assume that the pair $(X,\frb)$ is log canonical.
We consider the following relative version
of the definition of log canonical threshold (there is an analogous definition
in the language of $\Q$-divisors that is broadly used in the literature):
we define the \emph{mixed log canonical threshold} of $\fra$ with respect to
the pair $(X,\frb)$ to be
$$
\lct_{(X,\frb)}(\fra):=\sup\{c\geq 0\mid \text{$(X,\frb\.\fra^c)$ is log canonical}\}.
$$
Whenever the ambient variety $X$ is understood, we drop it from the notation,
and simply write $\lct_\frb(\fra)$.
Observe that in the case $\frb = \cO_X$,
the mixed log canonical threshold $\lct_{\cO_X}(\fra)$
is nothing else than the usual {\it log canonical threshold} $\lct(\fra)$ of $\fra$.
We make the convention 
$\lct_\frb(0) = 0$ and $\lct_\frb(\cO_X) = \infty$.

The fact that log canonicity can be checked on a log resolution allows us to describe
the mixed log canonical threshold in terms of any such resolution.
Suppose that $\pi\colon Y\to X$ is a log resolution of $\fra\cdot\frb$, and write
$\fra\cdot\cO_Y=\cO(-\sum_ia_iE_i)$, $\frb\cdot\cO_Y=\cO(-\sum_ib_iE_i)$, and
$K_{Y/X}=\sum_ik_iE_i$. Still assuming that $\fra$ and $\frb$ are nonzero ideals,
$\fra \ne \cO_X$, and $(X,\frb)$ is log canonical (that is, $\lct(\frb) \ge 1$),
it follows from the characterization of log canonicity in terms of a log resolution that
\begin{equation}\label{eq0}
\lct_{\frb}(\fra) =\min\left\{\frac{k_i+1-b_i}{a_i}\mid a_i>0\right\}.
\end{equation}
We see from the above formula that the mixed log canonical threshold is a rational number.
Note also that it is zero if and only if there is $i$ such that $a_i>0$ and $b_i = k_i+1$
(in other words, if $(X,\frb)$ is not Kawamata log terminal and there is a non-klt center
contained in the zero-locus of $\fra$).

It is convenient to use also a local version of the (mixed)  log canonical threshold.
For every point $p\in V(\fra)$ such that the pair $(X,\frb)$ is log canonical in some neighborhood of $p$, if in (\ref{eq0}) we take the minimum only over those $i$ such that
$p\in \pi(E_i)$, we get the \emph{mixed log canonical threshold at $p$}, denoted 
$\lct_{(X,\frb),p}(\fra)$. This is the maximum of $\lct_{\frb\vert_U}(\fra\vert_U)$, when $U$ ranges over the open neighborhoods of $p$. When $\frb=\cO_X$, we simply write
$\lct_p(\fra)$.

\begin{remark}\label{rem0}
It follows from the description in terms of a log resolution that if
$X=U_1\cup\ldots\cup U_r$, with
$U_j$ open, then $\lct_{\frb}(\fra)=\min_j\lct_{\frb\vert_{U_j}}(\fra\vert_{U_j})$.
\end{remark}

\begin{remark}\label{rem1}
If $\frb$ and $\fra$ are as above and $c: =\lct_{\frb}(\fra)$,
then $\lct(\frb\cdot \fra^{c})=1$ (where, of course, $\lct(\frb\cdot\fra^c)$ is the largest nonnegative $q$
such that the pair $(X,\frb^q\cdot\fra^{qc})$ is log canonical). Indeed, by assumption the pair
$(X,\frb\cdot\fra^c)$ is log canonical, and for every $\alpha>1$ the pair
$(X,(\frb\cdot\fra^c)^{\alpha})$ is not log canonical since
$(X,\frb\cdot\fra^{c\alpha})$ is not.
Note however that the converse of this property does not hold:
in fact, if $\lct(\frb) = 1$ and the zero-locus of $\fra$ does not contain
any non-klt center of $(X,\frb)$, then $c = \lct_\frb(\fra) > 0$
and $\lct(\frb\.\fra^t) = 1$ for every $0 < t \le c$.
\end{remark}

\begin{remark}\label{remark3}
Suppose that
 $X$, $\fra$ and $\frb$ are as above, with $X$ smooth. For every
 $p\in V(\fra)$, we have
$\lct_{(X,\frb),p}(\fra)=\lct_{(X',\frb')}(\fra')$, where 
$X'=\Spec(\widehat{\cO_{X,p}})$, and $\fra'$, $\frb'$ are the pull-backs of the ideals $\fra$
and, respectively, $\frb$ to $X'$. The argument follows as in the case $\frb=\cO_X$, for which we refer to
\cite[Proposition 2.9]{dFM}.
\end{remark}

We will adopt the following terminology.

\begin{definition}
Let $X$ and $\fra,\frb \subseteq \O_X$ be as above.
We say that a prime divisor $E$ over $X$ {\it computes $\lct_\frb(\fra)$}
if there is a log resolution $\pi \colon Y \to X$ such that, with the above notation,
$E$ induces the same valuation as a divisor $E_i$ on $Y$ for which $a_i > 0$ and
the minimum in \eqref{eq0} is achieved for this $i$.
\end{definition}

Suppose now that $k$ is algebraically closed.
For every $n \ge 0$, we consider the sets $\cT_n^{\rm sm}$, $\cT_n^{\rm quot}$,
$\cT_n^{\rm l.c.i}$, $\cM_n^{\rm sm}$ and $\cM_n^{\rm l.c.i.}$ defined
in the Introduction. Note that for $n=0$ all these sets are equal to $\{0\}$.
It is convenient to extend the definition to $n < 0$
by declaring all these sets to be empty in this range.
We will use the basic fact (cf. \cite[Proposition~3.3]{dFM}) that for every $n\geq 1$,
$$
\cT_n^{\rm sm} =
\{\lct_0(\fra) \mid \fra \subseteq (x_1,\ldots,x_n)\subset k[x_1,\ldots,x_n] \}.
$$
Similarly, for every $n\geq 1$ we have 
$$
\cM_n^{\rm sm} =
\{\lct_{({\mathbf A}^n,\frb),0}(\fra) \mid \fra,\frb \subseteq k[x_1,\dots,x_n],
\fra\subseteq (x_1,\ldots,x_n), \lct_0(\frb) \ge 1 \}.
$$
The proof is analogous to the non-mixed case, and is left to the reader.

\section{Effective $\frm$-adic semicontinuity of log canonical thresholds}

Let $X$ be a log canonical variety
defined over an algebraically closed field of characteristic zero $k$.
We start by proving Theorem~\ref{thm:intro:m-adic-semicont:ideals}
in the special case of principal ideals.

\begin{theorem}\label{thm:m-adic-semicont}
Let $E$ be a divisor over $X$, computing $\lct(f)$ for some $f\in\cO(X)$.
If $g\in\cO(X)$ is such that $\ord_E(f-g)>\ord_E(f)$, then after possibly replacing
$X$ by an open neighborhood of the center of $E$, we have $\lct(f)=\lct(g)$.
\end{theorem}

The interesting inequality is $\lct(g)\geq\lct(f)$, the reverse one being trivial.
Note that if the center of $E$ on $X$ is equal to a point $p\in X$, then
whenever $\mult_p(f-g) > \ord_E(f)$, we have $\ord_E(f-g)>\ord_E(f)$, and the theorem gives
$\lct_p(g)=\lct_p(f)$. 

As already explained in the Introduction, a proof of the theorem was given in \cite{Kol1} relying on 
deep results in the Minimal Model Program and on Inversion of Adjunction. We 
give an elementary proof, only using the Connectedness Theorem.

\begin{proof}[Proof of Theorem~\ref{thm:m-adic-semicont}]
The inequality $\lct(f)\geq\lct(g)$ is easy.
Indeed, since $\ord_E(f-g)>\ord_E(f)$, we have $\ord_E(g)=\ord_E(f)$, and therefore,
if $Y$ is the model over $X$ on which $E$ lies, then
$$
\lct(g)\leq \frac{\ord_E(K_{Y/X})+1}{\ord_E(g)}
=\frac{\ord_E(K_{Y/X})+1}{\ord_E(f)}=\lct(f).
$$

The first step in the proof of the reverse inequality
is to reduce to the case when $\ord_F(f-g)>\ord_F(f)$ for \emph{all}
divisors $F$ that compute $\lct(f)$ on some log resolution of $fg$.
In order to do this, let us choose
a log resolution $\pi \colon Y \to X$ of $fg(f-g)$
such that the divisor $E$ appears on $Y$.
Let $E_1,\dots,E_t$ be the irreducible components of the divisor
$K_{Y/X} + \pi^*(\Div(fg(f-g)))$.
After relabelling the indices, we may assume that $E=E_1$.
In the following, we denote
$$
a_i := \ord_{E_i}(f), \quad
b_i := \ord_{E_i}(g), \quad \text{and}\quad
k_i := \ord_{E_i}(K_{Y/X}).
$$

In order to prove the theorem, it is enough to show that for every $q\in \pi(E)$
we have $\lct_q(g)\geq\lct_q(f)$ (note that $\lct_q(f)=\lct(f)$). Fix such $q$. After possibly replacing 
$X$ by an open neighborhood of $q$, we may assume that $q\in\pi(E_i)$ for every $i$.

For every $m\geq 1$, we consider 
$f_m:=f^mh$ and $g_m:=g^mh$, where $h=f-g$. Note that by assumption
$\pi$ is a log resolution for both $f_m$ and $g_m$.

\begin{lemma}\label{lem:f_m-g_m}
If $m\gg 1$, then
\begin{enumerate}
\item[i)] $E_i$ computes $\lct(f_m)$ if and only if it computes $\lct(f)$ and, in addition, 
$$\frac{\ord_{E_i}(f)}{\ord_{E_i}(h)}=\min\left\{\frac{\ord_{E_j}(f)}{\ord_{E_j}(h)}\mid
E_j\,\text{computes}\,\lct(f)\right\}.$$
\item[ii)] For every $i$ such that $E_i$ computes $\lct(f_m)$, we have
$\ord_{E_i}(f_m-g_m)>\ord_{E_i}(f_m)$.
\end{enumerate}
\end{lemma}

\begin{proof}
We put $c_i=\ord_{E_i}(h)$. 
Since $m\gg 1$, we have
$$\frac{k_i+1}{a_i+\frac{c_i}{m}}\leq\frac{k_j+1}{a_j+\frac{c_j}{m}}$$
if and only if $\frac{k_i+1}{a_i}\leq\frac{k_j+1}{a_j}$, and either this inequality is strict, or
$\frac{k_i+1}{c_i}\leq \frac{k_j+1}{c_j}$.
This shows that every divisor $E_i$ that computes $\lct(f_m)$ also computes
$\lct(f)$. Furthermore, if $E_i$ computes $\lct(f)$, then it computes $\lct(f_m)$
if and only if $\frac{k_i+1}{c_i}\leq\frac{k_j+1}{c_j}$ for every $j$ such that 
$E_j$ computes $\lct(f)$. Note that this holds if and only if $\frac{a_i}{c_i}\leq
\frac{a_j}{c_j}$ (since $k_i+1=\lct(f)a_i$ and $k_j+1=\lct(f)a_j$), hence i).

Suppose now that $E_i$ computes $\lct(f_m)$. It follows from i) and our hypothesis that
$\frac{a_i}{c_i}\leq\frac{a_1}{c_1}<1$.
Since $f_m-g_m=(f^m-g^m)h$, in order to prove ii) it is enough to show that
$\ord_{E_i}(f^m-g^m)>m\cdot\ord_{E_i}(f)$. 
Note that $a_i<c_i$ implies $\ord_{E_i}(f)=\ord_{E_i}(g)$ (recall that $g=f-h$).
We write
$$
f^m-g^m=(g+h)^m-g^m=\sum_{\ell=1}^m{{m}\choose
{\ell}}h^{\ell}g^{m-\ell}.
$$
For every $\ell\geq 1$ we have
$\ord_{E_i}(h^{\ell}g^{m-{\ell}})>m\cdot\ord_{E_i}(f)$, hence
$\ord_{E_i}(f^m-g^m)>m\cdot\ord_{E_i}(f)$.
This completes the proof of the lemma.
\end{proof}

Observe that $\lct(f)=\lim_{m\to\infty} m\cdot\lct(f_m)$
and $\lct(g) =\lim_{m\to \infty} m\cdot\lct(g_m)$.
Indeed, it follows from definition that
$$
\lct(f_m)=\min_i\frac{k_i+1}{ma_i+c_i}
=\frac{1}{m}\cdot\min_i\frac{k_i+1}{a_i+\frac{c_i}{m}},
$$
which gives the first equality, and the second one follows in the same way.
Thus, if we can prove the theorem for $f_m$ and $g_m$ in place
of $f$ and $g$, for all $m \gg 1$, then we deduce the statement for $f$ and $g$.

Therefore, by Lemma~\ref{lem:f_m-g_m}, we are reduced to
proving Theorem~\ref{thm:m-adic-semicont} in the case when there is a log resolution
$\pi\colon Y\to X$ for $fg$ such that for all divisors $E_i$ on $\pi$ that compute
$\lct(f)$ we have $\ord_{E_i}(f-g)>\ord_{E_i}(f)$. We shall thus assume that
this is the case. We keep the notation previously introduced, so that
in particular $a_i = \ord_{E_i}(f)$ and $b_i = \ord_{E_i}(g)$ for every $i$.
Recall also that we may assume $q\in \pi(E_i)$ for all $i$. 

\begin{lemma}\label{lem:E_i-E_j}
Under the above assumptions,
if $E_i$ is a divisor computing $\lct(f)$, then
$\ord_{E_j}(f)=\ord_{E_j}(g)$ for every $j$ such that $E_i\cap E_j\neq\emptyset$.
\end{lemma}

\begin{proof}
Let $p \in E_i \cap E_j$ be a general point, and let $y_i, y_j \in \O_{Y,p}$
be part of a regular system of parameters, and generating the images in 
$\cO_{Y,p}$ of the ideals defining $E_i$ and $E_j$, respectively.
We have in $\cO_{Y,p}$
$$
\pi^*(f)=u y_i^{a_i} y_j^{a_j} \quad\text{and}\quad
\pi^*(g)=v y_i^{b_i} y_j^{b_j},
$$
where $u,v \in \O_{Y,p}$ are invertible elements.
By assumption, $\pi^*(f-g)=y_i^{a_i+1}w$ for some $w \in \O_{Y,p}$.
This has two consequences. The first is that $b_i=a_i$. Furthermore,
we see that $y_i^{-a_i}\pi^*(f)$ and $y_i^{-a_i}\pi^*(g)$ have the same restriction to
$E_i$. This implies that $b_j = a_j$, which is the assertion in the lemma.
\end{proof}

We can now finish the proof of Theorem~\ref{thm:m-adic-semicont}. Let $c=\lct(f)$,
and for every $i$ let
$$
\alpha_i := ca_i - k_i \quad\text{and}\quad
\beta_i := cb_i - k_i.
$$
Note that $\alpha_i \le 1$ for every $i$, and equality holds precisely for those $i$
such that $E_i$ computes $\lct(f)$.
The above lemma says that for every $i$ such that $\alpha_i=1$, we have
$\beta_i=1$, and more generally $\alpha_j = \beta_j$ for every $j$ such that
$E_i\cap E_j\neq\emptyset$.

To finish, we apply the main ingredient of the proof,
namely, the Connectedness Theorem
of Shokurov and Koll\'{a}r (see Theorem~7.4 in \cite{Kol2}), which in
our case says that the union $\cup_{\beta_j\geq 1}E_j$ is connected
in the neighborhood of $\pi^{-1}(q)$.
Since $q\in\pi(E_i)$ for every $i$, this implies that 
$\cup_{\beta_j\geq 1}E_j$ is connected.

Let us look at an arbitrary divisor $E_i$ that computes $\lct(f)$, so that
$\alpha_i=1$. We have seen that in this case $\beta_i=1$.
If $E_j$ is any other divisor that meets $E_i$ and such that $\beta_j\geq 1$, then
we have $1 \geq \alpha_j = \beta_j\geq 1$ by Lemma~\ref{lem:E_i-E_j}, and therefore
$\alpha_j = \beta_j = 1$. This implies by induction on $s$
that for every sequence of divisors $E_i,E_{j_1},\ldots,E_{j_s}$
such that any two consecutive divisors intersect, and such that
$\beta_{j_{\ell}}\geq 1$ for all $\ell$,
we have $\alpha_{j_{\ell}}=\beta_{j_{\ell}}=1$ for every $\ell$.
Since the set  $\cup_{\beta_j\geq 1}E_j$ is connected, 
we conclude that
$\beta_{j}\leq 1$ for every $j$, and thus $\lct(g)\geq c$.
This completes the proof of Theorem~\ref{thm:m-adic-semicont}.
\end{proof}

\begin{remark}
The above proof also gives the following statement. Suppose that 
$f$ and $g$ are as in Theorem~\ref{thm:m-adic-semicont}, such that for \emph{all}
divisors $E_i$ over $X$ computing $\lct(f)=c$, we have $\ord_{E_i}(f-g)>\ord_{E_i}(f)$
(it is easy to see that it is enough to check this condition only on the divisors on a fixed
log resolution of $f$). By the theorem, after restricting to an open neighborhood of
the non-klt locus of $(X,f^c)$ (this is the union of the centers of the divisors $E_i$
computing $\lct(f)$),
we have $\lct(g)=c$. In addition, the proof shows that
every divisor over $X$ that computes $\lct(g)$ also
computes $\lct(f)$.
\end{remark}

Theorem~\ref{thm:m-adic-semicont} can easily be extended to ideals,
as stated in Theorem~\ref{thm:intro:m-adic-semicont:ideals}, as follows.

\begin{proof}[Proof of Theorem~\ref{thm:intro:m-adic-semicont:ideals}]
We may assume that $X$ is affine.
Again, it is immediate
to see that the hypothesis implies that $\lct(\frb) \le \lct(\fra)$.
In order to prove the reverse inequality,
let $N$ be an integer larger than $\lct(\fra)$, and choose $N$ general
linear combinations $f_1,\dots,f_N$ of a fixed set of generators of $\fra$.
Note in particular that $\ord_E(f_i)= \ord_E(\fra)$ for all $i$.
Moreover, if $f:=f_1\dots f_N$, then $\lct(f)=\lct(\fra)/N$ and $E$ computes $\lct(f)$
(see, for example, \cite[Proposition 9.2.26]{positivity}). 

By assumption, we can write $f_i=g_i+h_i$, with $g_i\in\frb$ and $h_i\in \frq$.
Note that we have $\ord_E(h_i)>\ord_E(\fra)$, and hence $\ord_E(g_i)=\ord_E(\fra)$,
for every $i$. If $g:=g_1  \dots g_N$, then we can write
$$
f-g=h_1f_2\dots f_N + g_1h_2f_3\dots f_N + \dots + g_1g_2\dots g_{N-1}h_N.
$$
Since all terms in the above sum have order along $E$ larger than
$\ord_E(f)$, we conclude by Theorem~\ref{thm:m-adic-semicont} that
after possibly replacing $X$ by an open neighborhood of the center of $E$, we have
$\lct(g)\geq \lct(f)$.
Since $g\in\frb^N$, it follows that $\lct(\frb)\geq\lct(\fra)$.
\end{proof}

\begin{corollary}\label{formal_case}
Let $X=\Spec(R)$, where $R=k\llbracket x_1,\ldots,x_n\rrbracket$, and let 
$\fra$ and $\frb$ proper ideals in $R$. Suppose that $E$ is a divisor over $X$ with center
equal to the closed point, such that $E$ computes $\lct(\fra)$. If $\frb+\frq=\fra+\frq$, where
$\frq=\{h\in R\mid\ord_E(h)>\ord_E(\fra)\}$, then $\lct(\frb)=\lct(\fra)$.
\end{corollary}

\begin{proof}
 It is enough to show that $\lct(\frb+\frm^N)=\lct(\fra+\frm^N)$
for all $N\gg 0$, where
$\frm$ denotes the maximal ideal in $R$
(we use the fact that $\lct(\frb)=\lim_{N\to\infty}\lct(\frb+\frm^N)$ and 
$\lct(\fra)=\lim_{N\to\infty}\lct(\fra+\frm^N)$, see \cite[Proposition 2.5]{dFM}).
Since the center of $E$ is equal to the closed point, there is a divisor $F$ over ${\mathbf A}^n$
with center the origin such that $E$ is obtained from $F$  by base-change with respect to
$\Spec(R)\to {\mathbf A}^n$. If $\widetilde{\fra}_N:=(\fra+\frm^N)\cap k[x_1,\ldots,x_n]$ and 
$\widetilde{\frb}_N:=(\frb+\frm^N)\cap k[x_1,\ldots,x_n]$,
then $\fra+\frm^N=\widetilde{\fra}_N\cdot R$ and $\frb+\frm^N=\widetilde{\frb}_N\cdot R$.
Hence
$\lct(\fra+\frm^N)=\lct_0(\widetilde{\fra}_N)$ and $\lct(\frb+\frm^N)=\lct_0(\widetilde{\frb}_N)$
(see, for example, \cite[Corollary 2.8]{dFM}).

On the other hand, we have $\lct(\fra+\frm^N)\geq \lct(\fra)$ for every $N$, and
$\lct(\fra+\frm^N)\leq\lct(\fra)$ for $N>\ord_E(\fra)$. It follows that for such $N$ we have
$\lct(\fra+\frm^N)=\lct(\fra)$, and furthermore, $E$ computes $\lct(\fra+\frm^N)$. 
Therefore $F$ computes $\lct_0(\widetilde{\fra}_N)$.
If $N>\ord_E(\fra)$, then $\ord_F(\widetilde{\fra}_N)=\ord_E(\fra)$, and 
$$(x_1,\ldots,x_n)^N\subseteq\widetilde{\frq}:=\{h\in k[x_1,\ldots,x_n]\mid
\ord_F(h)>\ord_F(\widetilde{\fra}_N)\}=\frq\cap k[x_1,\ldots,x_n].$$
We deduce that $\widetilde{\frb}_N+\widetilde{\frq}=\widetilde{\fra}_N+\widetilde{\frq}$, hence by
 Theorem~\ref{thm:intro:m-adic-semicont:ideals}
 we have $\lct_0(\widetilde{\frb}_N)=\lct_0(\widetilde{\fra}_N)$.
 We conclude that $\lct(\frb+\frm^N)=\lct(\fra+\frm^N)$ for all $N\gg 0$, and therefore
 $\lct(\frb)=\lct(\fra)$.
\end{proof}

\section{Generic limits of sequences of ideals}\label{sect:gen-limits}

In this section we review the construction from \cite{Kol1}, extending it
from sequences of power series to sequences of ideals.
In fact, we will need a version dealing with several such sequences.
The goal is to associate to these sequences of ideals in a fixed polynomial ring
or ring of power series, corresponding ``limit" ideals through a sequence
of $\frm$-adic approximations and field extensions.

For the sake of notation we only treat the case of two
sequences. This is the only case needed in the paper. It will be
however clear that the construction can be carried out for any given
number of sequences. We also note that by 
 taking the second sequence to be constant,
the construction given below reduces
in particular to a construction of generic limits for just one
sequence. Furthermore, the assertions in Proposition~\ref{prop:gen-limit:lct-E}
and Corollary~\ref{cor:lct=lim} below reduce to statements about one sequence
by taking $q=0$.

Let $R = k \llbracket x_1,\dots,x_n\rrbracket$
be the ring of formal power series in $n$ variables
with coefficients in an algebraically closed field $k$, and
let $\frm$ be its maximal ideal.
If $k\subset K$ is a field extension, then we put
$R_K :=  K \llbracket x_1,\dots,x_n\rrbracket$
and $\frm_K:= \frm\. R_K$.

For every $d \ge 1$, we consider the quotient homomorphism $R \to R/\frm^d$.
We identify the ideals in $R/\frm^d$ with the ideals in $R$ containing $\frm^d$.
Let $\cH_d$ be the Hilbert scheme parametrizing the ideals in $R/\frm^d$, with the
reduced scheme structure. Since $\dim_k(R/\frm^d)<\infty$, $\cH_d$ is an algebraic variety.
Note that  for every field extension $K$ of $k$, the $K$-valued points of
$\cH_d\times\cH_d$ correspond to pairs 
of ideals in $R_K$ containing $\frm_K^d$. 
Mapping a pair of ideals in $R/\frm^{d}$ to the pair consisting of their images in $R/\frm^{d-1}$ gives a surjective map
$t_d\colon\cH_d\times\cH_d\to \cH_{d-1}\times\cH_{d-1}$. This is not a morphism. However, by Generic Flatness
we can cover $\cH_d\times\cH_d$ by disjoint locally closed subsets  such that
the restriction of $t_d$ to each of these subsets is a morphism. In particular, for every
irreducible closed subset $Z\subseteq\cH_d\times\cH_d$, the map $t_d$ induces a rational map
$Z\rat \cH_{d-1}\times\cH_{d-1}$.

Suppose now that $(\fra_i)_{i \in I_0}$ and $(\frb_i)_{i\in I_0}$
are sequences of ideals in $R$ indexed by the set $I_0 = \Z_+$.
We consider sequences of irreducible closed subsets $Z_d\subseteq \cH_d\times\cH_d$ for $d\geq 1$
such that 
\begin{enumerate}
\item[$(\star)$] For every $d\geq 1$, the projection $t_{d+1}$ induces a dominant rational map $\phi_{d+1}\colon
Z_{d+1}\rat Z_{d}$.
\item[$(\star\star)$] For every $d\geq 1$, there are infinitely many $i$ with 
$(\fra_i+\frm^d,\frb_i+\frm^d)\in Z_d$, and the set of such 
$(\fra_i+\frm^d,\frb_i+\frm^d)$ is dense in $Z_d$.
\end{enumerate}
Given such a sequence $(Z_d)_{d\geq 1}$, 
we define inductively nonempty  open subsets $Z^\o_d\subseteq
Z_d$, and a nested sequence of infinite subsets
$$I_0\supseteq I_1\supseteq I_2\supseteq\cdots,$$ 
as follows. We put $Z^\o_1=Z_1$ and
$I_1=\{i\in I_0\mid(\fra_i+\frm,\frb_i+\frm)\in Z_1^{\o}\}$. For $d\geq 2$, let $Z^\o_d=\phi_d^{-1}(Z^\o_{d-1})
\subseteq {\rm Domain}(\phi_d)$ and $I_d=\{i\in I_0\mid (\fra_i+\frm^d,
\frb_i+\frm^d)\in Z^\o_d\}$.
It follows by induction on $d$ that $Z^\o_d$ is open in $Z_d$, and condition $(\star\star)$ implies
that each $I_d$ is infinite. Furthermore, it is clear that $I_d\supseteq I_{d+1}$.

Sequences $(Z_d)_{d\geq 1}$ satisfying $(\star)$ and $(\star\star)$ can be constructed as follows. 
We first choose a minimal irreducible closed subset $Z_1 \subseteq \cH_1\times\cH_1$ with the property that it contains $(\fra_i + \frm,\frb_i+\frm)$ for infinitely many indices $i \in I_0$.
We set $J_1 = \{ i \in I_0 \mid (\fra_i + \frm,\frb_i+\frm) \in Z_1 \}$.
By construction, $J_1$ is an infinite set and $Z_1$ is the closure of
$\{(\fra_i+\frm,\frb_i+\frm)\mid i\in I_1\}$.
Next, we choose a minimal closed subset $Z_2 \subseteq \cH_2\times\cH_2$ that contains
$(\fra_i + \frm^2,\frb_i+\frm^2)$ for infinitely many $i$ in $J_1$ (note that by minimality, $Z_2$ is irreducible).
 By construction, the set
 $J_2=\{i\in J_1\mid (\fra_i+\frm^2,\frb_i+\frm^2)\in Z_2\}$ is infinite, and 
$Z_2$ is the closure of $\{(\fra_i+\frm^2,\frb_i+\frm^2)\mid i\in J_2\}$.
As we have seen, $t_2$ induces a rational map
$\phi_2\colon Z_2 \rat Z_1$.
Note that by the minimality
in the choice of $Z_1$, the rational map $\phi_2$ is dominant.
Repeating this process we select a sequence $(Z_d)_{d\geq 1}$ that satisfies
$(\star)$ and $(\star\star)$ above. 

Suppose now that we have a sequence $(Z_d)_{d\geq 1}$ with these two properties.
The rational maps $\phi_d$ induce a nested sequence of function fields $k(Z_d)$. Let
$K:=\bigcup_{d \ge 1} k(Z_d)$.
Each morphism $\Spec(K)\to Z_d\subseteq\cH_d\times\cH_d$ 
corresponds to a pair of ideals $\fra'_d$ and $\frb'_d$ in $R_K$ containing $\frm_K^d$, and the compatibility between these morphisms
implies that there are (unique) ideals $\fra$ and $\frb$ in $R_K$ such that 
 $\fra'_d=\fra+\frm_K^d$ and $\frb'_d=\frb+\frm_K^d$ for all $d$.

\begin{definition}
With the above notation,
we say that the pair of ideals $(\fra,\frb)$ is {\it a generic limit}
of the sequence of pairs of ideals $(\fra_i,\frb_i)_{i \ge 1}$.
\end{definition}

\begin{remark}
The reader may compare the above construction with a similar one that can be used to show
that every sequence $(x_i)_{i \geq 1}$, with all $x_i$ in a closed bounded interval $L_0=[a,b]$,
contains a convergent subsequence. In that case, one also constructs by induction 
closed bounded intervals $L_d=[u_d,w_d]$ with $L_d\subseteq L_{d-1}$ and
$(w_d-u_d)<\epsilon_d$ (for some sequence $\epsilon_d$ converging to zero), and infinite subsets
$I_d\subseteq I_{d-1}\subseteq I_0=\Z_+$, such that $x_i\in L_d$ for all $i\in I_d$. 
With this notation, it is then clear that $(x_i)_{i\geq 1}$ contains a subsequence converging to
$\sup_du_d=\inf_dw_d$.
\end{remark}

We list in the next lemma some easy properties of generic limits.
The proof is straightforward, so we omit it.

\begin{lemma}\label{lemma:gen-limit:basic-properties}
Let $(\fra_i)_{i\geq 1}$ and $(\frb_i)_{i\geq 1}$ be sequences of ideals in $R$, and let 
$(\fra,\frb)$  be a generic limit as above, with $\fra,\frb\subseteq R_K$. 
\begin{enumerate}
\item[i)] If $\frb_i = \frc$ for every $i$, where $\frc \subseteq R$
is a fixed ideal, then $\frb=\frc\cdot R_K$. 
\item[ii)]  If $q\geq 1$ is such that $\fra_i \subseteq \frm^q$ for every $i$, then $\fra \subseteq \frm^q_K$. 
\item[iii)] If $q\geq 1$ is such that $\fra_i\not\subseteq\frm^q$ for every $i$, then 
$\fra\not\subseteq\frm_K^q$.
\item[iv)] If $\fra=(0)$, then for every $q\geq 1$ there are infinitely many $d$ such that 
$\fra_d\subseteq\frm^q$. 
\end{enumerate}
\end{lemma}

In the following proposition we keep the notation used in the definition of generic limits.
Recall that we have also defined the nested sequence of infinite sets $(I_d)_{d\geq 1}$. 

\begin{proposition}\label{prop:gen-limit:lct-E}
Let $\fra,\frb \subseteq R_K$ be such that $(\fra,\frb)$ is a generic limit
of the sequence $(\fra_i, \frb_i)_{i\geq 1}$ of pairs of ideals in $R$.
Assume that $\fra_i,\frb_i \ne R$ for all $i$. 
For every $d$ there is an infinite subset
$I_d^\o \subseteq I_d$ such that for all nonnegative integers $p$ and $q$
$$
\lct((\fra + \frm_K^d)^p\cdot (\frb+\frm_K^d)^q)
 = \lct((\fra_i + \frm^d)^p\cdot(\frb_i+\frm^d)^q)
\quad\text{for every $i \in I_d^\o$.}
$$
Moreover, if $E$ is a divisor over $\Spec(R_K)$, with center at the closed point and computing
$\lct(\fra^p\cdot\frb^q)$ for some nonnegative integers $p$ and $q$, then there is an integer 
$d_E$ such that
for every $d \ge d_E$ the following holds: there is an infinite subset
$I_d^E \subseteq I_d^\o$, and for every $i\in I_d^E$
a divisor $E_i$ over $\Spec(R)$ computing $\lct((\fra_i+\frm^d)^p\cdot (\frb_i+\frm^d)^q)$, such that
$\ord_E(\fra+\frm_K^d)=\ord_{E_i}(\fra_i+\frm^d)$ and $\ord_{E}(\frb+\frm_K^d)=
\ord_{E_i}(\frb_i+\frm_K^d)$.
\end{proposition}

In the second assertion in the proposition, both $d_E$ and the sets $I_d^E$
also depend on $p$ and $q$, while $E_i$ also depends on $d$.

\begin{proof}
Note that every ideal of the form $\frc+\frm^d$ can be considered as the ideal of 
a scheme on $\AAA^n_k$ supported at the origin, and the log canonical threshold
computed in $\Spec(R)$ is the same as when computed in $\AAA_k^n$ 
(cf. \cite[Corollary~2.8]{dFM}). Of course, the same holds if we replace $k$ by $K$.
Whenever we can, we
adopt this alternative point of view, since base change 
works better in this setting (by base change an
affine space becomes another affine space).

On $\cH_d\times\AAA_k^n$ we have the universal family of ideals $\cI$.
Pulling this back via the two projections $\cH_d\times\cH_d\times\AAA_k^n\to
\cH_d\times\AAA_k^n$,
and then restricting to $Z_d\times\AAA_k^n$ gives the ideals $\cI_d$ and 
$\cJ_d$ on $Z_d\times\AAA_k^n$.
Let $\mu_d \colon Y_d \to Z_d \times \AAA_k^n$ be 
a log resolution of the product $\cI_d\cdot\cJ_d$, and let
${\mathcal E}$ be the relevant simple normal crossings divisor on $Y_d$.
By Generic Smoothness, there is a nonempty open subset
$U_d \subseteq Z_d$ such that the induced map $Y_d\to Z_d$ is smooth over $U_d$, and furthermore,
${\mathcal E}$ has relative simple normal crossings over $U_d$. 
In this case, the fiber of $Y_d\to Z_d$ over a point in $U_d$ corresponding to a pair of 
ideals $(\frc_1,\frc_2)$
gives a log resolution of the ideal $\frc_1\cdot \frc_2$ in $\AAA^n_k$.
It follows that for every $p$ and $q$, the log canonical threshold
$\lct(\frc_1^p\cdot\frc_2^q)$ is independent of the point $(\frc_1,\frc_2)\in U_d$. Moreover, it is equal to the similar log canonical threshold computed for the pair of ideals parametrized by the generic point of $Z_d$.
These are ideals in $k(Z_d)[x_1,\ldots,x_n]$ whose extensions to 
$K[x_1,\ldots,x_n]$ are $\fra+\frm^d_K$ and $\frb+\frm_K^d$.
We thus take $I_d^\o \subset I_d$ to consist of those $i$ for which
$(\fra_i + \frm^d,\frb_i+\frm^d)$ is in $U_d$. 
Condition $(\star\star)$ on the sequence $(Z_d)_{d\geq 1}$ implies that 
$I_d^\o$ is an infinite set.

For the second assertion in the proposition, observe first that 
since $E$ has center equal to the closed point, 
there is a divisor $F$
over $\AAA_K^n$ with center at the origin, such that $E$ is obtained from $F$ by 
base-change with respect to $\Spec(R_K)\to\AAA_K^n$. 
Given an ideal $\frc+\frm_K^d\subset R_K$, the divisor $E$ computes the log canonical threshold
of this ideal if and only if $F$ computes the log canonical threshold of the corresponding ideal
in $K[x_1,\ldots,x_n]$.

Note that the divisor $F$,
a priori defined over $K$, is in fact defined over a subextension $L$ of $K/k$,
of finite type over $k$.
Let $d_E>\ord_E(\fra^p\cdot\frb^q)$ be an integer such that
 $F$ is defined over
$k(Z_{d_E})$.
For $d\geq d_E$, we have $\lct((\fra+\frm_K^d)^p\cdot (\frb+\frm_K^d)^q)=
\lct(\fra^p\cdot\frb^q)$, and $E$ computes both these log canonical thresholds: for this one argues as in the beginning of the proof of Theorem~\ref{thm:m-adic-semicont},
observing that in this case we have $\lct(\fra^p\cdot\frb^q)\leq\lct((\fra+\frm_K^d)^p\cdot
(\frb+\frm_K^d)^q)$
due to the inclusion
 $\fra^p\cdot\frb^q\subseteq(\fra+\frm_K^d)\cdot(\frb+\frm_K^d)^q$.

On the other hand, for every such $d$ we can find a nonempty open subset $W_d \subseteq Z_d$
and a log resolution $\nu_d\colon Y_d'\to W_d\times \AAA_k^n$ 
of the restriction of $\cI_d\cdot\cJ_d$ to $W_d\times\AAA^n_k$,
such that $F$  
is obtained from a divisor ${\mathcal F}'$ on $Y_d'$
by base-change with respect to the composition 
$$\AAA_K^n\to \AAA_{k(Z_d)}^n
\to W_d\times\AAA_k^n.$$ 
Arguing as in the first part of the proof, we see that after possibly replacing 
$W_d$ by a smaller open subset, we may assume that 
$Y_d'$ is smooth over $W_d$, and furthermore, that the relevant divisor ${\mathcal E}'$
has relative simple normal crossings over $W_d$. Note that ${\mathcal F}'$ is a component of 
${\mathcal E}'$.

Let $I_d^E:=\{i \in I_d^\o \mid (\fra_i + \frm^d,\frb_i+\frm^d) \in W_d\}$. Again, condition $(\star\star)$ on
the sequence $(Z_d)_{d\geq 1}$
implies that $I_d^E$ is infinite.
Since $F$ computes the log canonical threshold of the (extension to $K[x_1,\ldots,x_n]$
of the) suitable product corresponding to the pair of ideals parametrized
by the
generic point of $W_d$, it follows that if $i\in I_d^E$, and $F_i$ is a connected component of 
the fiber of ${\mathcal F}'$ over the point in $W_d$ representing $(\fra_i+\frm^d,\frb_i+\frm^d)$, then
$F_i$ computes $\lct((\fra_i+\frm^d)^p\cdot (\frb_i+\frm^d)^q)$. Moreover, we have
$\ord_F(\fra+\frm_K^d)=\ord_{F_i}(\fra_i+\frm^d)$ and 
$\ord_F(\frb+\frm_K^d)=\ord_{F_i}(\frb_i+\frm^d)$. If $E_i$ is obtained from $F_i$ by base-change
via $\Spec(R)\to \AAA^n_k$, then $E_i$ satisfies the requirement in the proposition.
\end{proof}

\begin{corollary}\label{cor:lct=lim}
With the above notation, for every sequence $(i_d)_{d \ge 1}$ with
$i_d \in I_d^\o$, we have $\lct(\fra^p\cdot \frb^q) = 
\lim_{d \to \infty} \lct(\fra_{i_d}^p\cdot\frb_{i_d}^q)$ for all nonnegative integers $p$ and $q$.
In particular, if the sequence $(\lct(\fra_i^p\cdot\frb_i^q))_{i \ge 1}$ is convergent, then it
converges to $\lct(\fra^p\cdot\frb^q)$.
\end{corollary}

\begin{proof}
Recall the following basic fact: if $\frc$ and $\frc'$ are proper ideals in $R$, 
with $\frc+\frm^d=\frc'+\frm^d$, then
$$|\lct(\frc)-\lct(\frc')|\leq \frac{n}{d}$$
(see \cite[Corollary~2.10]{dFM}). Note that this equality also holds when $\frc$
or $\frc'$ are zero. Of course, a similar result holds for ideals in $R_K$. It follows from 
Proposition~\ref{prop:gen-limit:lct-E}
that for every $d\geq 1$ we have
$$|\lct(\fra^p\cdot \frb^q)-\lct(\fra_{i_d}^p\cdot\frb_{i_d}^q)|
\leq |\lct(\fra^p\cdot\frb^q)-\lct((\fra+\frm_K^d)^p\cdot (\frb+\frm_K^d))|$$
$$+
|\lct((\fra_{i_d}+\frm^d)^p\cdot (\frb_{i_d}+\frm^d)^q)-\lct(\fra_{i_d}^p\cdot\frb_{i_d}^q)|
\leq\frac{2n}{d}.$$
The assertion in the proposition is an immediate consequence.
\end{proof}

\begin{remark}
It is clear that both the construction and the above results generalize in an obvious way to any finite number 
of sequences of ideals.
\end{remark}

\section{Log canonical thresholds on smooth varieties}

This section is devoted to the proof of Theorem~\ref{thm:intro:T_n^sm}.
For completeness, we also include the proof of the smooth case of Koll\'ar's Accumulation
Conjecture \cite{Kol2}, which is already known by the results in \cite{dFM,Kol1}:
the case of limits of decreasing sequences was first treated in \cite{dFM},
and the proof was completed in \cite{Kol1} where the the case of
(potential) limits of increasing sequences was also treated.

\begin{theorem}
\label{thm:T_n^sm}
For every $n$, the set $\cT_n^{\rm sm}$ satisfies the ascending chain condition,
and its set of accumulation points is $\cT_{n-1}^{\rm sm}$.
\end{theorem}

We start with an easy lemma that can be used to replace an ideal by another ideal with the same log canonical threshold, and such that this log canonical threshold is computed by a divisor having a zero-dimensional center.

\begin{lemma}
Let $\fra$ be an ideal contained in the maximal ideal $\frm_K$ of $K\llbracket x_1,\ldots,x_n\rrbracket$.
We put $q := \max\{ t \ge 0 \mid \lct(\fra\.\frm_K^t) = \lct(\fra) \}$.
\begin{enumerate}
\item[i)] We have $q\in\QQ_{\geq 0}$.
\item[ii)] If we write $q=r/s$, for nonnegative integers $r$ and $s$, then $\lct(\fra^s\cdot\frm_K^r)=\frac{\lct(\fra)}{s}$, and this log canonical threshold is computed by a divisor with center equal to the closed point.
\item[iii)] We have $q=0$ if and only if $\lct(\fra)$ is computed by a divisor 
with center over the closed point.
\end{enumerate}
\end{lemma}

\begin{proof}
Let $\pi\colon Y\to X=\Spec\left(K\llbracket x_1,\ldots,x_n\rrbracket\right)$ be a log resolution of 
$\fra\cdot\frm_K$, and write $\fra\cdot\cO_Y=\cO(-\sum_ia_iE_i)$, $\frm_K\cdot\cO_Y=\cO_Y(-\sum_ib_iE_i)$, and $K_{Y/X}=\cO_Y(-\sum_ik_iE_i)$.
Let $I$ denote the set of those $i$ for which  $E_i$ has center equal to
 the closed point, that is, such that $b_i>0$.
 
 Let $c=\lct(\fra)$. Note that we have $\lct(\fra\cdot\frm_K^t)\leq c$ for every $t\geq 0$. Furthermore, 
 $\lct(\fra\cdot\frm_K^t)\geq c$ if and only if 
 $$k_i+1\geq c(a_i+tb_i)$$
 for all $i$. If $i\not\in I$, then $b_i=0$ and this inequality holds for all $t$. We conclude that 
 $$q=\min\left\{\frac{k_i+1-ca_i}{cb_i}\mid i\in I\right\}.$$
 This shows that $q\in\QQ$. Moreover, if $i\in I$ is such that this minimum is achieved, then
 $E_i$ computes $\lct(\fra^s\cdot\frm_K^r)$, and $E_i$ has center equal to the closed point.
 The assertion in iii) is clear.
\end{proof}

\begin{proof}[Proof of Theorem~\ref{thm:T_n^sm}]
Let $(c_i)_{i \ge 1}$ be a strictly monotone sequence with terms in $\cT_n^{\rm sm}$,
and let $c = \lim_{i \to \infty} c_i$ (the limit is finite,
since $\cT_n^{\rm sm}$ is bounded above by $n$). For every $i$ we can select an ideal
$\widetilde{\fra}_i  \subseteq (x_1,\ldots,x_n)\subset k[x_1,\dots,x_n]$ with 
$\lct_0(\widetilde{\fra}_i) = c_i$.
Let $\fra_i=\widetilde{\fra}_i\cdot k\llbracket x_1,\ldots,x_n\rrbracket$.
Consider a generic limit $(\fra,\frb)$ of the sequence of pairs of ideals $(\fra_i,\frm)_{i\geq 1}$,
constructed as in Section~\ref{sect:gen-limits},
with $\fra,\frb\subseteq K\llbracket x_1,\ldots,x_n\rrbracket$. Note that by 
Lemma~\ref{lemma:gen-limit:basic-properties}, we have $\fra\subseteq\frm_K$, and
$\frb=\frm_K$. 
Since $\lct(\fra_i)=\lct_0(\widetilde{\fra}_i)$ (see, for example, \cite[Proposition 2.9]{dFM}), it follows from
Corollary~\ref{cor:lct=lim} that $\lct(\fra) = c$. If $c=0$, then the sequence $(c_i)_{i\geq 1}$ can't be strictly increasing. Furthermore, we have $0\in \cT_{n-1}^{\rm sm}$ (note indeed that
$n>0$), hence this case is clear, and we may assume that $c>0$. In particular, $\fra\neq (0)$.

Let $q$ be the rational number attached to $\fra$ as in the lemma, and write $q=r/s$, 
with $r$ and $s$ nonnegative integers.
By construction, we have
$$
\lct(\fra^s\.\frm_K^r) = \frac 1s \lct(\fra).
$$
On the other hand, we certainly have
$$
\lct(\fra_i^s\cdot\frm^r) \leq \frac 1s \lct(\fra_i) \quad\text{for every $i$}.
$$
Note in particular that if $(c_i)_{i \ge 1}$ is a strictly increasing sequence, then
$\lct(\fra_i^s\cdot\frm^r) < \lct(\fra^s\cdot\frm^r)$ for every $i$.

By the choice of $q$, $\lct(\fra^s\cdot\frm_K^r)$ is computed by a divisor
$E$ which lies over the closed point of $\Spec(K\llbracket x_1,\dots,x_n\rrbracket)$.
Fix any $d \geq d_E$, with $d_E$ associated to $E,s,r$ and to the sequence $(\fra_i,\frm)_{i\geq 1}$ by
Proposition~\ref{prop:gen-limit:lct-E}. As in the proof of that proposition, we may and will assume that
$d_E>\ord_E(\fra^s\cdot\frm_K^r)$, so that for all $d\geq d_E$ we have 
$\lct(\fra^s\cdot\frm_K^r)=\lct((\fra+\frm_K^d)^s\cdot \frm_K^r)$,
and $E$ computes both log canonical thresholds.

By Proposition~\ref{prop:gen-limit:lct-E}, there is an infinite set $I_d^E \subseteq \Z_+$
such that for every $i\in I_d^E$ we have 
$\lct((\fra+\frm_K^d)^s\cdot\frm_K^r) = \lct((\fra_i+ \frm^d)^s\cdot\frm^r)$, and moreover, there is
a divisor $E_i$ over 
${\rm Spec}\left(k\llbracket x_1,\ldots,x_n\rrbracket\right)$ computing
$\lct((\fra_i+\frm^d)^s\cdot\frm^r)$, and such that
$$\ord_{E_i}((\fra_i+\frm^d)^s\cdot\frm^r)=\ord_E((\fra+\frm_K^d)^s\cdot\frm_K^r)=
\ord_E(\fra^s\cdot\frm_K^r).$$
Since $E_i$ is a divisor computing $\lct((\fra_i+\frm^d)^s\cdot\frm^r)$, its center is equal to the closed point.
Furthermore, by our condition on $d$ we have
$$\ord_{E_i}(\frm^d)\geq d>\ord_E(\fra^s\cdot\frm^r)=\ord_{E_i}((\fra_i+\frm^d)^s\cdot\frm^r),$$
hence Corollary~\ref{formal_case}
gives for every $i\in I_d^E$
$$\lct(\fra_i^s\cdot\frm^r)=\lct((\fra_i+\frm^d)^s\cdot\frm^r)=\lct((\fra+\frm_K^d)^s\cdot\frm_K^r)=\lct(\fra^s\cdot\frm_K^r).$$

It follows from the above discussion that $(c_i)_{i \ge 1}$ cannot be a
strictly increasing sequence,
which proves that $\cT_n^{\rm sm}$ satisfies the ascending chain condition.
By exclusion, $(c_i)_{i \ge 1}$ has to be a strictly decreasing sequence.
Since the sequence $(\lct(\fra^s_i\cdot\frm^r))_{i \ge 1}$
has repeating terms, we deduce that $q > 0$. Equivalently,
$\lct(\fra)$ is not computed by any divisor with center at the closed point.
Therefore, if $F$ is a divisor over $\Spec(K\llbracket x_1,\dots,x_n\rrbracket)$
computing $\lct(\fra)$,
then the center of $F$ in $\Spec(K\llbracket x_1,\dots,x_n\rrbracket)$
is positive dimensional,
and hence, after localizing at its generic point,
we see that $\lct(\fra) \in \cT_{n-1}^{\rm sm}$
(cf. \cite[Propositions~2.11 and 3.1]{dFM}).
As it is easy and well-known that, conversely, every element in
$\cT_{n-1}^{\rm sm}$ is an accumulation point of $\cT_n^{\rm sm}$,
we conclude that $\cT_{n-1}^{\rm sm}$ is equal to
the set of accumulation points of $\cT_n^{\rm sm}$.
\end{proof}

The following proposition allows us to reduce log canonical thresholds on
varieties with quotient singularities to log canonical thresholds on smooth
varieties. We say that a variety $X$ has
{\it quotient singularities} at $p\in X$ if there is a smooth variety $U$, a finite
group $G$ acting on
$U$, and a point $q\in V=U/G$ such that the two completions
$\widehat{\cO_{X,p}}$ and $\widehat{\cO_{V,q}}$ are isomorphic as $k$-algebras.
We say that $X$ has quotient singularities if it has quotient singularities at every point.

In the above definition,
one can assume that $U$ is an affine space and that the action of $G$
is linear. Furthermore, one can assume that $G$
acts with no fixed points in codimension one
(otherwise, we may replace $G$ by $G/H$ and $U$ by $U/H$, where $H$ is generated
by all pseudoreflections in $G$, and by Chevalley's theorem \cite{Chevalley}, the quotient
$U/H$ is again an affine space).
Using Artin's approximation results (see Corollary~2.6 in \cite{Artin}), it follows that
there is an \'{e}tale neighborhood of $p$ that is also an \'{e}tale neighborhood of $q$.
In other words, there is a variety $W$, a point $r\in W$, and \'{e}tale maps
$\phi\colon W\to X$ and $\psi\colon W\to V$, such that $p=\phi(r)$ and $q=\psi(r)$.
After replacing $\phi$ by the composition
$$
W\times_VU\to W\overset{\phi}\to X,
$$
we may assume that in fact we have an \'{e}tale map $U/G\to X$ containing $p$ in its image,
with $U$ smooth,
and such that $G$ acts on $U$ without fixed points in codimension one.
This reinterpretation of the definition of quotient singularities seems to be well-known
to experts, but we could not find an explicit reference in the literature.

\begin{proposition}\label{reduction_quotient}
Let $X$ be a variety with quotient singularities, and let $\fra$ be a proper
nonzero ideal on $X$.
For every $p$ in the zero-locus $V(\fra)$ of $\fra$, there is a smooth variety $U$,
a nonzero ideal
$\frb$ on $U$, and a point $q$ in $V(\frb)$ such that $\lct_p(X,\fra)=\lct_q(U,\frb)$.
\end{proposition}

\begin{proof}
Let us choose an \'{e}tale map $\phi\colon U/G\to X$ with $p\in {\rm Im}(\phi)$, where
$U$ is a smooth variety, and $G$ is a finite group acting on $U$
without fixed points in codimension one.
Let $\widetilde{\phi}\colon U\to X$ denote the composition of $\phi$ with the quotient map.
Since $G$ acts without fixed points in codimension one, $\widetilde{\phi}$ is
\'{e}tale in codimension one, hence $K_U=\widetilde{\phi}^*(K_X)$. It follows from
Proposition~5.20 in \cite{KM} that if $\frb=\fra\cdot\cO_U$, then the pair $(X,\fra^t)$
is log canonical if and only if the pair
$(U,\frb^t)$ is log canonical (actually the result in \emph{loc. cit.} only covers
the case when $\fra$
is locally principal, but one can easily reduce to this case, by taking a suitable
product of general
linear combinations of the local generators of $\fra$). We conclude that there is a
point $q\in
V(\frb)$ such that $\lct_p(X,\fra)=\lct_q(U,\frb)$.
\end{proof}

It follows that $\cT_n^{\rm quot} = \cT_n^{\rm sm}$ for every $n$,
and therefore we deduce by Theorem~\ref{thm:T_n^sm} that
Shokurov's ACC Conjecture and Koll\'ar's Accumulation Conjecture
hold for log canonical thresholds on varieties with quotient singularities.

\begin{corollary}\label{quotient}
For every $n$, the set $\cT_n^{\rm quot}$ satisfies the ascending chain condition
and its set of accumulation points is equal to $\cT_{n-1}^{\rm quot}$.
\end{corollary}

\begin{remark}\label{usual_definition}
At least over the complex numbers, one usually says
that $X$ has quotient singularities at $p$
if the germ of analytic space $(X,x)$ is isomorphic to
$M/G$, where $M$ is a complex manifold,
and $G$ is a finite group acting on $M$. It is not hard to check that in this
context this definition is equivalent with the one we gave above.
\end{remark}

\section{Log canonical thresholds on l.c.i. varieties}

In this section we prove that the ACC Conjecture
holds for log canonical thresholds (and mixed log canonical thresholds)
on l.c.i. varieties, and prove Theorem~\ref{thm:intro:M_n^lci}.
We start with the case of mixed log canonical thresholds on smooth varieties.

\begin{theorem}\label{thm:M_n^sm}
For every $n$, the set $\cM_n^{\rm sm}$ satisfies the ascending chain condition.
\end{theorem}

\begin{proof}
Suppose that $\cM_n^{\rm sm}$ contains a strictly increasing sequence $(c_i)_{i \ge 1}$.
Let $c=\lim_{i \to \infty} c_i$ (which is finite, since $\cM_n^{\rm sm}$ is bounded above by $n$).
We can find ideals $\widetilde{\fra}_i,\,\widetilde{\frb}_i\subseteq k[x_1,\dots,x_n]$, 
 with $\widetilde{\fra}_i \subseteq (x_1,\ldots,x_n)$
and $\lct_0(\widetilde{\frb}_i) \ge 1$, such that $c_i=\lct_{({\mathbf A}^n,\widetilde{\frb}_i),0}
(\widetilde{\fra_i})$.
If $\fra_i$ and $\frb_i$ are the ideals generated by $\widetilde{\fra}_i$ and, respectively, 
$\widetilde{\frb}_i$ in $k\llbracket x_1,\ldots,x_n\rrbracket$, then $c_i=\lct_{\frb_i}(\fra_i)$
by Remark~\ref{remark3}. 
Consider a generic limit $(\fra,\frb)$ of the sequence $(\fra_i,\frb_i)_{i\geq 1}$, constructed as in
Section~\ref{sect:gen-limits}, with $\fra,\frb\subseteq K\llbracket x_1,\ldots,x_n\rrbracket$.
By Corollary~\ref{cor:lct=lim}, $\lct(\frb)$ is a limit point of the sequence 
$(\lct(\frb_i))_{i\geq 1}$, hence $\lct(\frb)\geq 1$.
Therefore $c':=\lct_\frb(\fra)$ is well defined.

Consider first any positive integers $p$ and $q$ such that $p/q<c$. By assumption, we have
$c_i>p/q$ for all $i\gg 1$. Let $X=\Spec\left(k\llbracket x_1,\ldots,x_n\rrbracket\right)$.
The pair $(X,\frb_i\.\fra_i^{p/q})$ is log canonical, hence
$\lct(\frb_i^q\.\fra_i^p)\geq 1/q$, for all $i\gg 1$.
It follows from Corollary~\ref{cor:lct=lim} that there
is a sequence $(i_d)_{d \ge 1}$ in $\Z_+$ such that
$$
\lct(\frb^q\.\fra^p) = \lim_{d \to \infty} \lct(\frb_{i_d}^q\.\fra_{i_d}^p).
$$
This implies in particular that $\lct(\frb^q\.\fra^p) \geq 1/q$, and
therefore $c'\geq p/q$. As this holds for every $p/q<c$,
we conclude that $c'\geq c$.

On the other hand, since $c'\in\QQ$, we may write $c'=r/s$
for positive integers $r$ and $s$. It follows from Remark~\ref{rem1}
that $\lct(\frb\.\fra^{r/s})=1$, and thus $\lct(\frb^s\.\fra^r)=1/s$.
Applying again Corollary~\ref{cor:lct=lim},
we find a sequence $(j_d)_{d \ge 1}$ in $\Z_+$ such that
$$
\lct(\frb^s\.\fra^r) = \lim_{d \to \infty} \lct(\frb_{j_d}^s\.\fra_{j_d}^r).
$$
The fact that $\cT_n^{\rm sm}$ satisfies
the ascending chain condition (cf. Theorem~\ref{thm:T_n^sm})
implies that there are infinitely many $d$ such that
$\lct(\frb_{j_d}^s\.\fra_{j_d}^r)\geq 1/s$, hence
$\lct_{\frb_{j_d}}(\fra_{j_d})\geq r/s$.
For any such $d$ we have
$$
c' \geq c > c_{j_d}\geq \frac{r}{s}=c',
$$
which is a contradiction.
\end{proof}

In order to extend the above result to the case
of ambient varieties with l.c.i. singularities,
we use the following application of
Inversion of Adjunction. This is the key tool that allows us to
replace mixed log canonical thresholds on locally complete intersection varieties
with the similar type of invariants on ambient smooth varieties.

\begin{proposition}\label{inversion}
Let $A$ be a smooth irreducible variety over $k$, and $X\subset A$
a closed subvariety of pure codimension $e$, that is normal and
locally a complete intersection.
Suppose that $\frb$ and $\fra$ are ideals on $A$, with $\fra\neq\cO_A$, and such that
$X$ is not contained in the union of the zero-loci of $\frb$ and $\fra$.
\begin{enumerate}
\item[i)]
The pair $(X,\frb\vert_X)$ is log canonical if and only if
for some open neighborhood $U$ of $X$,  the pair
$(U,\frb\cdot \frp^e\vert_U)$ is log canonical, where
$\frp$ is the ideal defining $X$ in $A$.
\item[ii)]
If $(X,\frb\vert_X)$ is log canonical, and if $X$
intersects the zero-locus  of $\fra$, then for some open neighborhood
$V$ of $X$ we have
$$
\lct_{\frb\vert_X}(X,\fra\vert_X)=\lct_{\frb\vert_V\cdot \frp^e\vert_V} (V, \fra\vert_V).
$$
\end{enumerate}
\end{proposition}

\begin{proof}
Both assertions follow from Inversion of Adjunction (see Corollary~3.2 in \cite{EM3}), as
this says that
for every nonnegative $q$, the pair $(X,(\frb\.\fra^q)\vert_X)$ is log canonical
if and only if the pair $(A,\frb\.\fra^q\.\frp^e)$ is log canonical in some neighborhood of
$X$.
\end{proof}

The next fact, which must be well-known to the experts,
allows us to control the dimension of the ambient variety in the process of
replacing a mixed log canonical threshold on an l.c.i. variety by one on a smooth variety.
Given a closed point $x\in X$,
we denote by $T_xX$ the Zariski tangent space of $X$ at $x$.

\begin{proposition}\label{bound}
Let $X$ be a locally complete intersection variety.
If $X$ is log canonical, then $\dim_kT_xX\leq 2\dim X$ for every $x\in X$.
\end{proposition}

\begin{proof}
Fix $x\in X$, and let $N=\dim\,T_xX$. After possibly replacing $X$ by an
open neighborhood of $x$,
we may assume that we have a closed embedding of $X$ in a smooth irreducible variety $A$,
of codimension $e$, with
$\dim A=N$. If $X=A$, then $N=\dim X$ and we are done.

Suppose now that $e\geq 1$.
Since $X$ is locally a complete intersection, it follows from Inversion of Adjunction
(see Corollary~3.2 in \cite{EM3}) that the pair $(A,\frp^e)$ is log canonical,
where $\frp$ is the ideal
of $X$ in $A$. In particular, if $E$ is the exceptional divisor of
the blow-up $A'$ of $A$ at $x$, and ${\rm ord}_E$ is the corresponding valuation, then
we have
$$
N=1+\ord_E(K_{A'/A})\geq e\cdot \ord_E(\frp)\geq 2e=2(N-\dim X).
$$
This gives $N\leq 2\dim X$.
\end{proof}

We are now ready to prove Theorem~\ref{thm:intro:M_n^lci},
and hence Corollary~\ref{cor:intro:T_n^lci}.

\begin{proof}[Proof of Theorem~\ref{thm:intro:M_n^lci}]
By Theorem~\ref{thm:M_n^sm}, we know that $\cM_n^{\rm sm}$ satisfies
the ascending chain condition for every $n$.
Then it is clear that in order to prove that $\cM_n^{\rm l.c.i.}$
also satisfies the ascending chain condition for every $n$, it suffices to show that
$$
\cM_n^{\rm l.c.i.}\subseteq\cM_{2n}^{\rm sm}.
$$
Suppose that $(X,\frb)$ is log canonical, with $X$ locally a complete
intersection of dimension $ n$, and let $c=\lct_{\frb}(\fra)$.
Let $x\in X$ be any point in the center of a divisor computing $\lct_{\frb}(\fra)$.
For every open neighborhood $U$ of $x$ we have $\lct_{\frb\vert_U}(U,\fra\vert_U)=c$.
Since $X$ is log canonical, it follows from
Proposition~\ref{bound} that $\dim_kT_xX\leq 2n$.
After replacing $X$ by a suitable neighborhood of $x$,
we may assume that there is
a closed embedding $X\hookrightarrow A$, where $A$ is a smooth variety of
dimension $2n$. Proposition~\ref{inversion} implies that after possibly
replacing $A$ by a
neighborhood of $X$, we have $c=\lct_{\frb_1\cdot\frp^{e}}(\fra_1)$, where $\frp$
is the ideal defining $X$ in $A$, $e$ is the codimension of $X$ in $A$, and $\frb_1$
and $\fra_1$ are ideals in $A$
whose restrictions to $X$ give, respectively, $\frb$ and $\fra$.
Thus $c \in \cM_{2n}^{\rm sm}$.
\end{proof}

\begin{proof}[Proof of Corollary~\ref{cor:intro:T_n^lci}]
It follows by Theorem~\ref{thm:intro:M_n^lci}, since
$\cT_n^{\rm l.c.i} \subseteq \cM_n^{\rm l.c.i}$.
\end{proof}

\providecommand{\bysame}{\leavevmode \hbox \o3em
{\hrulefill}\thinspace}

\end{document}